\documentclass[12pt]{article}
\usepackage{amsmath,amsfonts,amssymb,amsthm}
\DeclareMathOperator{\hypk}{Kl_2}
\DeclareMathOperator{\Mod}{mod}

\newcommand{\Z}{\mathbb{Z}}
\newcommand{\Q}{\mathbb{Q}}

\renewcommand{\pmod}[1]{\,(\Mod{#1})}
\newcommand{\twosum}[2]{\sum_{\substack{#1\\#2}}}
\def\sumstar{\sideset{}{^\star}\sum}
\newtheorem{thm}{Theorem}[section]
\newtheorem{lem}[thm]{Lemma}
\begin{document}
\title{The divisor function in arithmetic progressions to smooth moduli}
\author{A.J. Irving\\
Mathematical Institute, Oxford}
\date{}
\maketitle

\begin{abstract}
By using the $q$-analogue of van der Corput's method we study the divisor function in an arithmetic progression to modulus $q$.  We show that the expected asymptotic formula holds for a larger range of $q$ than was previously known, provided that $q$ has a certain factorisation.
\end{abstract}

\section{Introduction}

Given an arithmetic function $f(n)$ it is natural to consider the sum 
$$\twosum{n\leq x}{n\equiv a\pmod q}f(n).$$
For many functions $f$ we might hope to show that when $(a,q)=1$ this is asymptotic to 
$$\frac{1}{\varphi(q)}\twosum{n\leq x}{(n,q)=1}f(n).$$
In applications it is often essential that we establish such a result uniformly in $q\leq x^\theta$ with $\theta$ as large as possible.  

In this paper we will consider the divisor  function $\tau(n)$, which counts the number of positive divisors of $n$.  We therefore let 
$$D(x,q,a)=\twosum{n\leq x}{n\equiv a\pmod q}\tau(n)$$
and 
$$D(x,q)=\frac{1}{\varphi(q)}\twosum{n\leq x}{(n,q)=1}\tau(n).$$
We then wish to estimate $E(x,q,a)=D(x,q,a)-D(x,q)$.  We hope to show that for some $\delta>0$ we have the bound 
\begin{equation}\label{requiredbound}
E(x,q,a)\ll \frac{x^{1-\delta}}{q}.
\end{equation}

If $q\leq x^{2/3-\eta}$ for some $\eta>0$ then (\ref{requiredbound}) holds with a $\delta$ depending on $\eta$.  This was proved independently in unpublished work of Hooley, Linnik and Selberg, it is a consequence of the Weil bound for Kloosterman sums.  For larger $q$, no nontrivial bound is known for individual $E(x,q,a)$ but there are various results on average.  For example Fouvry \cite[Corollaire 5]{fouvrytitchmarsh} showed that for any $\eta,A>0$ and any $a\in \Z$ we have 
$$\twosum{x^{2/3+\eta}\leq q\leq x^{1-\eta}}{(q,a)=1}|E(x,q,a)|\ll_{A,a,\eta} x(\log x)^{-A}.$$
An average over moduli $x^{2/3-\eta}\leq q\leq x^{2/3+\eta}$ was considered by Fouvry and Iwaniec in \cite{fouiwdivisor}.  Their approach requires them to work only with moduli $q$ which have a squarefree factor $r$ of a certain size.  Specifically, they show that if $r$ is squarefree with $r\leq x^{3/8}$ and $(r,a)=1$ then for any $\eta>0$ we have 
$$\twosum{rs^2\leq x^{1-6\eta}}{(s,ar)=1}|E(x,rs,a)|\ll_\eta r^{-1}x^{1-\eta}.$$
Observe that to handle moduli $q=rs$ of size $x^{2/3}$ with this result it is necessary that $r\geq x^{1/3+6\eta}$.  Further results are possible if we exploit averaging over the residue class $a\pmod q$.  See for example Banks, Heath-Brown and Shparlinski \cite{bhbshpar} and Blomer \cite{blomer}.

We will show that (\ref{requiredbound}) holds for an individual $E(x,q,a)$ for $q$ almost as large as $x^{\frac{55}{82}}$ provided that $q$ factorises in a certain way.  This will follow by optimising the sizes of the parameters in the following result.

\begin{thm}\label{divisorthm}
Suppose that $q=q_0q_1q_2q_3$ is squarefree and $(a,q)=1$.  For any $x\geq q$, $\delta\in (0,\frac{1}{12})$ and any $\epsilon>0$ we have
$$E(x,q,a)\ll_\epsilon q^{-1}x^{1-\delta+\epsilon}+x^{2\delta+\epsilon}
\left(\sum_{j=1}^3x^{2^{-j-1}}q^{1/2-2^{-j}}q_{4-j}^{2^{-j}}+x^{1/16}q^{3/8}q_0^{1/16}+q^{1/2}q_0^{-1/16}\right).$$ 
\end{thm}

It is not immediately clear when the estimate in this theorem is nontrivial.  We therefore prove the following, in which we exploit the fact that if $q$ is sufficiently smooth then we can find a suitable factorisation for which  our bound is close to optimal.

\begin{thm}\label{divisorsmooth}
Suppose $\varpi,\eta>0$ satisfy 
$$246\varpi+18\eta<1.$$
There exists a $\delta>0$, depending on $\varpi$ and $\eta$, such that for any $x^\eta$-smooth, squarefree $q\leq x^{2/3+\varpi}$ and any $(a,q)=1$ we have 
$$E(x,q,a)\ll_{\varpi,\eta}q^{-1}x^{1-\delta}.$$
\end{thm}

Observe that for any $\varpi<\frac{1}{246}$ this theorem shows that there is an $\eta>0$ for which the conclusion holds.  This means that we get a bound for sufficiently smooth $q$ which are almost as large as $x^{\frac{55}{82}}$.  The smoothness assumption is not necessary, it is simply a convenient way of guaranteeing that suitably sized factors exist.  For example, given a squarefree $q\sim x^{2/3}$, Theorem \ref{divisorthm} gives a nontrivial estimate provided that, for some $\eta>0$, we have  $q=q_0q_1q_2q_3$ with 
$$x^\eta\leq q_0\leq x^{1/3-\eta}$$
and
$$q_j\leq x^{1/6-\eta}\text{ for }1\leq j\leq 3.$$

Writing 
$$e_q(x)=e^{\frac{2\pi i x}{q}},$$
the proof of Theorem \ref{divisorthm} depends on estimates for short Kloosterman sums 
$$\twosum{n\in I}{(n,q)=1}e_q(b\overline n),$$
where $b$ is an integer and $I$ is an interval of length $O(\sqrt{x})$.
If $(b,q)=1$ then the Weil bound gives an estimate of $O_\epsilon(q^{1/2+\epsilon})$ for such a sum.  For the sizes of $x$ and $q$ in which we are interested this is a significant saving over the trivial bound of $\sqrt{x}$.  In particular it is enough to estimate $E(x,q,a)$ if $q\leq x^{2/3-\eta}$.  For larger $q$ we must improve upon the Weil estimate.  This is achieved for special $q$ by means of the following result.

\begin{thm}\label{shortkloost}
Let $q=q_0q_1\ldots q_l$ be squarefree.  Suppose that $(a,q)=1$ and that $I$ is an interval of length at most $N\leq q$.  Let 
$$S=\twosum{n\in I}{(n,q)=1}e_q(a\overline n).$$
For any $\epsilon>0$ we have 
$$S\ll_{\epsilon,l} q^\epsilon\left(\sum_{j=1}^lN^{2^{-j}}q^{1/2-2^{-j}}q_{l-j+1}^{2^{-j}}+N^{2^{-l}}q^{1/2-2^{-l}}q_0^{1/2^{l+1}}+q^{1/2}q_0^{-1/2^{l+1}}\right).$$
\end{thm}

This theorem is very similar to that of Heath-Brown \cite[Theorem 2]{rhbx32}.  His result can be applied to our sum $S$ to obtain the bound 
$$S\ll_{\epsilon,l} q^\epsilon\left(\sum_{j=1}^lN^{1-2^{-j}}q_{l-j+1}^{2^{-j}}+N^{1-2^{-l}}q_0^{1/2^{l+1}}+Nq_0^{-1/2^{l+1}}\right).$$
When the sizes of the factors $q_j$ are chosen optimally this result of Heath-Brown is nontrivial provided that $N$ is approximately $q^{\frac{1}{l+1}}$.  In contrast, our bound is most useful when $N\approx q^{1-\frac{1}{l+1}}$ in which case it can improve on the Weil bound.  

As in Heath-Brown's work our proof of Theorem \ref{shortkloost} uses the $q$-analogue of van der Corput's method.  We begin by completing the sum $S$ and then apply the differencing process $l$-times, whereas Heath-Brown applied differencing directly to $S$.  The result is a sum of products of $2^l$ Kloosterman sums which we estimate by another completion followed by the application of a bound for complete exponential sums due to Fouvry, Ganguly, Kowalski and Michel \cite{fgkm}.  In other words, our result is a $q$-analogue of the $BA^lB$ van der Corput estimate whereas Heath-Brown's is analogous to $A^lB$. A $q$-analogue of $BA^2B$ was used by Heath-Brown in \cite{rhbcw} but the exponential sums in that work are not Kloosterman sums. 

The assumption that $q$ is squarefree is important for two reasons.  Firstly, it guarantees that the factors $q_j$ are coprime in pairs, thereby avoiding many unpleasant technicalities. Secondly, it means that we need only consider complete exponential sums to prime moduli.  To handle $q$ which are not squarefree Lemma \ref{completeexp} would have to be generalised to prime-power moduli.     

Throughout this work we use the notation $x\sim y$ for the inequality $y\leq x<2y$.  We adopt the standard convention that $\epsilon$ denotes a sufficiently small positive quantity whose value may differ at each occurrence.  

\subsection*{Acknowledgements}

This work was completed as part of my DPhil, for which I was funded by EPSRC grant EP/P505666/1. I am very grateful to the
EPSRC for funding me and to my supervisor, Roger Heath-Brown, for all
his valuable help and advice.  I would also like to thank Emmanuel Kowalski for his assistance with the complete exponential sums arising in this work. 

\section{Proof of Theorem \ref{divisorthm}}

In this section we will show that Theorem \ref{divisorthm} follows from Theorem \ref{shortkloost}.  Recall that we wish to estimate 
$$E(x,q,a)=\twosum{uv\leq x}{uv\equiv a\pmod q}1-\frac{1}{\varphi(q)}\twosum{uv\leq x}{(uv,q)=1}1.$$
By a dyadic subdivision it is enough to consider each of the $O((\log x)^2)$ sums of the form 
$$E_1(U,V,q,a)=\twosum{u\sim U,v\sim V}{uv\leq x,uv\equiv a\pmod q}1-\frac{1}{\varphi(q)}\twosum{u\sim U,v\sim V}{uv\leq x,(uv,q)=1}1=D_1(U,V,q,a)-D_1(U,V,q),$$
say. We must bound $E_1(U,V,q,a)$ for all $U,V\geq 1$ for which $UV\leq x$. However, by symmetry we can assume that $U\leq\sqrt{x}$.    

We will use a short interval decomposition to remove the constraint $uv\leq x$ from $D_1(U,V,q,a)$ and $D_1(U,V,q)$.  Specifically we divide the range $u\sim U$ into $O(x^\delta)$ intervals of length $Ux^{-\delta}$ and the range $v\sim V$ into $O(x^\delta)$ intervals of length $Vx^{-\delta}$.  We will denote the resulting intervals by 
$$I_1(U_1)=[U_1,U_1+Ux^{-\delta})$$
and 
$$I_2(V_1)=[V_1,V_1+Vx^{-\delta}).$$
We only need consider the case that $U_1V_1\leq x$.  Dropping the constraint $uv\leq x$ has the effect of including in the above sums points $(u,v)\in I_1(U_1)\times I_2(V_1)$ with 
$$x<uv\leq (U_1+Ux^{-\delta})(V_1+Vx^{-\delta})\leq x+O(x^{1-\delta}).$$
It follows that the errors introduced by removing the constraint are bounded by 
$$\twosum{x<n\leq x+O(x^{1-\delta})}{n\equiv a\pmod q}\tau(n)\ll_\epsilon q^{-1}x^{1-\delta+\epsilon}$$
and
$$\frac{1}{\varphi(q)}\twosum{x<n\leq x+O(x^{1-\delta})}{(n,q)=1}\tau(n)\ll_\epsilon q^{-1}x^{1-\delta+\epsilon}.$$
We conclude that it is enough 
to bound  $O(x^{2\delta}(\log x)^2)$ sums of the form  
$$E_2(U_1,V_1,q,a)=D_2(U_1,V_1,q,a)-D_2(U_1,V_1,q)$$
where 
$$D_2(U_1,V_1,q,a)=\#\{u\in I_1(U_1),v\in I_2(V_1):uv\equiv a\pmod q\}$$
and
$$D_2(U_1,V_1,q)=\frac{1}{\varphi(q)}\#\{u\in I_1(U_1),v\in I_2(V_1):(uv,q)=1\}.$$
Specifically we have 
$$E(x,q,a)\ll_\epsilon q^{-1}x^{1-\delta+\epsilon}+x^{2\delta+\epsilon}\max_{U_1,V_1}|E_2(U_1,V_1,q,a)|.$$ 

We now write
\begin{eqnarray*}
D_2(U_1,V_1,q,a)&=&\twosum{u\in I_1(U_1),v\in I_2(V_1)}{uv\equiv a\pmod q}1\\
&=&\twosum{u\in I_1(U_1)}{(u,q)=1}\twosum{v\in I_2(V_1)}{v\equiv a\overline u\pmod q}1\\
&=&\frac{1}{q}\twosum{u\in I_1(U_1)}{(u,q)=1}\sum_{v\in I_2(V_1)}\sum_{k=1}^q e_q(k(a\overline u-v))\\
&=&\frac{1}{q}\sum_{k=1}^q \left(\twosum{u\in I_1(U_1)}{(u,q)=1}e_q(ak\overline u)\right)\left(\sum_{v\in I_2(V_1)}e_q(-kv)\right).\\
\end{eqnarray*}
The $k=q$ terms in this are 
$$\frac{1}{q}\#\{u\in I_1(U_1):(u,q)=1\}\#I_2(V_1).$$
On the other hand 
\begin{eqnarray*}
D_2(U_1,V_1,q)&=&\frac{1}{\varphi(q)}\twosum{u\in I_1(U_1)}{(u,q)=1}\twosum{v\in I_2(V_1)}{(v,q)=1}1\\
&=&\frac{1}{\varphi(q)}\twosum{u\in I_1(U_1)}{(u,q)=1}\left(\frac{\varphi(q)}{q}\#I_2(V_1)+O_\epsilon(q^\epsilon)\right)\\
&=&\frac{1}{q}\#\{u\in I_1(U_1):(u,q)=1\}\#I_2(V_1)+O_\epsilon(q^{-1+\epsilon}x^{1/2}),\\
\end{eqnarray*}
where we have used our assumption that $U\leq \sqrt{x}$.  Since $\delta<\frac{1}{6}$ and $q\leq x$ we have 
$$x^{2\delta}\cdot q^{-1+\epsilon}x^{1/2}<q^{-1}x^{1-\delta+\epsilon}$$ 
so we conclude that the $k=q$ terms correspond to $D_2(U_1,V_1,q)$ with a sufficiently small error.

It remains to bound 
$$\frac{1}{q}\sum_{k=1}^{q-1} \left|\twosum{u\in I_1(U_1)}{(u,q)=1}e_q(ak\overline u)\right|\left|\sum_{v\in I_2(V_1)}e_q(-kv)\right|.$$
We write this as 
$$\frac{1}{q}\sum_{d|q}\twosum{k=1}{(k,q)=d}^{q-1}\left|\twosum{u\in I_1(U_1)}{(u,q)=1}e_q(ak\overline u)\right|\left|\sum_{v\in I_2(V_1)}e_q(-kv)\right|$$
$$=\frac{1}{q}\twosum{d|q}{d<q}\;\sumstar_{k\pmod {q/d}}\left|\twosum{u\in I_1(U_1)}{(u,q)=1}e_{q/d}(ak\overline u)\right|\left|\sum_{v\in I_2(V_1)}e_{q/d}(-kv)\right|.$$
However, since $q$ is squarefree we have 
\begin{eqnarray*}
\twosum{u\in I_1(U_1)}{(u,q)=1}e_{q/d}(ak\overline u)&=&\twosum{u\in I_1(U_1)}{(u,q/d)=1}e_{q/d}(ak\overline u)\sum_{e|(d,u)}\mu(e)\\
&=&\sum_{e|d}\mu(e)\twosum{u\in I_1(U_1)/e}{(u,q/d)=1}e_{q/d}(ak\overline {eu}).\\
\end{eqnarray*}
Our sum is therefore bounded by 
$$\frac{1}{q}\twosum{d|q}{d<q}\;\sumstar_{k\pmod {q/d}}\left|\sum_{v\in I_2(V_1)}e_{q/d}(-kv)\right|\sum_{e|d}\left|\twosum{u\in I_1(U)/e}{(u,q/d)=1}e_{q/d}(ak\overline {eu})\right|.$$
We have the standard estimate 
$$\sum_{v\in I_2(V_1)}e_{q/d}(-kv)\ll \min\left(Vx^{-\delta},\frac{1}{\|dk/q\|}\right)$$
so that this is at most 
$$\frac{1}{q}\twosum{d|q}{d<q}\sum_{e|d}\max_{(b,q/d)=1}\left|\twosum{u\in I_1(U)/e}{(u,q/d)=1}e_{q/d}(b\overline u)\right|\sumstar_{k\pmod {q/d}}\frac{1}{\|dk/q\|}$$
$$\ll_\epsilon q^{\epsilon}\twosum{d|q}{d<q}\frac{1}{d}\sum_{e|d}\max_{(b,q/d)=1}\left|\twosum{u\in I_1(U)/e}{(u,q/d)=1}e_{q/d}(b\overline u)\right|.$$
To estimate the contribution to this from 
$d\geq qx^{-2/3+2\delta}$
we apply the Weil bound which gives 
$$\max_{(b,q/d)=1}\left|\twosum{u\in I_1(U_1)/e}{(u,q/d)=1}e_{q/d}(b\overline {u})\right|\ll_\epsilon \frac{Ux^{-\delta}d}{qe}+(q/d)^{1/2+\epsilon}.$$
The contribution to our sum from such $d$ is therefore bounded by 
$$q^{\epsilon}\twosum{d|q}{qx^{-2/3+2\delta}\leq d<q}(\frac{Ux^{-\delta}}{q}+q^{1/2}/d^{3/2})\ll_\epsilon q^\epsilon(Ux^{-\delta}q^{-1}+q^{-1}x^{1-3\delta}).$$
The contribution of these $d$ to $E(x,q,a)$ is therefore $O_\epsilon(q^{-1}x^{1-\delta+\epsilon})$. If $q< x^{2/3-2\delta}$ then this analysis covers all values of $d$ and therefore completes the proof.  

If $q\geq x^{2/3-2\delta}$ and $d< qx^{-2/3+2\delta}$ we apply Theorem \ref{shortkloost} with $l=3$ and the factorisation 
$$\frac{q}{d}=\prod_{j=0}^3\frac{q_j}{(q_j,d)},$$
which holds since $q$ is squarefree.  We have 
$$\frac{q}{d}\geq x^{2/3-2\delta}\geq \sqrt{x},$$
since $\delta<\frac{1}{12}$.  We may therefore deduce that if $(b,q/d)=1$ then    
$$\frac{1}{d}\twosum{u\in I_1(U)/e}{(u,q/d)=1}e_{q/d}(b\overline {u})\ll_\epsilon 
q^\epsilon\left(\sum_{j=1}^3x^{2^{-j-1}}q^{1/2-2^{-j}}q_{4-j}^{2^{-j}}+x^{1/16}q^{3/8}q_0^{1/16}+q^{1/2}q_0^{-1/16}\right).$$
It follows that we have
$$q^{\epsilon}\twosum{d|q}{d<qx^{-2/3+2\delta}}\frac{1}{d}\sum_{e|d}\max_{(b,q/d)=1}\left|\twosum{u\in I_1(U)/e}{(u,q/d)=1}e_{q/d}(b\overline u)\right|$$
$$\ll_\epsilon q^\epsilon\left(\sum_{j=1}^3x^{2^{-j-1}}q^{1/2-2^{-j}}q_{4-j}^{2^{-j}}+x^{1/16}q^{3/8}q_0^{1/16}+q^{1/2}q_0^{-1/16}\right).$$
We conclude that the contribution of this to $E(x,q,a)$ is majorised by 
$$x^{2\delta+\epsilon}\left(
\sum_{j=1}^3x^{2^{-j-1}}q^{1/2-2^{-j}}q_{4-j}^{2^{-j}}
+x^{1/16}q^{3/8}q_0^{1/16}+q^{1/2}q_0^{-1/16}\right).$$
This completes the proof of Theorem \ref{divisorthm}.

\section{Proof of Theorem \ref{divisorsmooth}}

Suppose  $\varpi,\eta,q$ and $a$ are as in Theorem \ref{divisorsmooth}.  Let $\delta>0$ be a parameter which we will eventually choose to be very small.  We may suppose that $q\geq x^{2/3-2\delta}$ since the result is known for smaller $q$.  Applying Theorem \ref{divisorthm} we deduce that for any $\epsilon>0$ we have 
$$E(x,q,a)\ll_\epsilon q^{-1}x^{1-\delta+\epsilon}+x^{2\delta+\epsilon}\left(\sum_{j=1}^3x^{2^{-j-1}}q^{1/2-2^{-j}}q_{4-j}^{2^{-j}}+x^{1/16}q^{3/8}q_0^{1/16}+q^{1/2}q_0^{-1/16}\right).$$ 
The first term in this is sufficiently small.  We optimise the remaining terms by working with a factorisation for which $q_j\approx Q_j$ with 
$$Q_0=q^{-2/15}x^{1/3},$$
$$Q_1=q^{-1/15}x^{1/6},$$
$$Q_2=q^{7/15}x^{-1/6}$$
and 
$$Q_3=q^{11/15}x^{-1/3}.$$
Observe that $Q_0Q_1Q_2Q_3=q$ and that for all sufficiently small $\delta$ we have $Q_j>x^{1/18}>x^\eta$ for all $j$.  Since $q$ is $x^\eta$-smooth we may  find a factorisation $q=q_0q_1q_2q_3$ with 
$$q_1\in [Q_1x^{-\eta/5},Q_1x^{4\eta/5}],$$
$$q_2\in [Q_2x^{-3\eta/5},Q_2x^{2\eta/5}],$$
$$q_3\in [Q_3x^{-4\eta/5},Q_3x^{\eta/5}]$$
so that 
$$q_0\in [Q_0x^{-7\eta/5},Q_0x^{8\eta/5}].$$
This gives 
$$E(x,q,a)\ll_\epsilon q^{-1}x^{1-\delta+\epsilon}+x^{2\delta+\epsilon}\left(x^{1/12+\eta/10}q^{11/30}+x^{-1/48+7\eta/80}q^{61/120}\right).$$ 
Finally, recalling that $q\leq x^{2/3+\varpi}$ we get 
\begin{eqnarray*}
E(x,q,a)&\ll_\epsilon&q^{-1}x^{1-\delta+\epsilon}+q^{-1}x^{2\delta+\epsilon}\left(x^{1/12+\eta/10}q^{41/30}+x^{-1/48+7\eta/80}q^{181/120}\right)\\
&\ll_\epsilon&q^{-1}x^{1-\delta+\epsilon}+q^{-1}x^{2\delta+\epsilon}\left(x^{\frac{179+18\eta+246\varpi}{180}}+x^{\frac{709+63\eta+1086\varpi}{720}}\right).\\ 
\end{eqnarray*}
We know that 
$$246\varpi+18\eta<1.$$
In particular $\varpi<\frac{1}{246}$ and $\eta<\frac{1}{18}$ so 
$$63\eta+1086\varpi<\frac{649}{82}<8.$$
Theorem \ref{divisorsmooth} therefore follows on taking $\delta$ and $\epsilon$ sufficiently small in terms of $\varpi$ and $\eta$.

\section{Proof of Theorem \ref{shortkloost}}

Suppose that for some $1\leq j\leq l$ we have $[q/N]<q_{l-j+1}$.  Then 
$$q^\epsilon N^{2^{-j}}q^{1/2-2^{-j}}q_{l-j+1}^{2^{-j}}\gg q^{1/2+\epsilon}.$$
Our result therefore follows from the Weil bound.  We may therefore assume, for the remainder of the paper, that $[q/N]\geq q_j$ for all $1\leq j\leq l$.   

\subsection{Completion of $S$}

Let $f(k)$ be the Fourier transform of the interval $I$:
$$f(k)=\sum_{n\in I}e_q(-nk).$$
We have 
$$S=\frac{1}{q}\sum_{k\pmod q}f(k)S(a,k,q)$$
where $S(a,k,q)$ is the Kloosterman sum given by 
$$S(a,k,q)=\sumstar_{n\pmod q}e_q(a\overline n+kn).$$
Since $f(0)\ll N$ and $S(a,0,q)=\mu(q)\ll 1$ we get 
$$S\ll \frac{N}{q}+\frac{1}{q}\left|\sum_{k\ne 0\pmod q}f(k)S(a,k,q)\right|.$$
The term $N/q$ is clearly small enough.

We may assume that $I\subseteq [M,M+N)$ for some integer $M$.  We then write 
$$f(k)=e_q(-kM)\twosum{n<N}{n+M\in I}e_q(-kn)=e_q(-kM)g(k),$$
say.  Thus 
$$S\ll \frac{N}{q}+\frac{1}{q}\left|\sum_{k\ne 0\pmod q}g(k)e_q(-kM)S(a,k,q)\right|.$$
We will consider the contribution to this bound from $0<k\leq
q/2$. One can use a completely analogous treatment for the range $-q/2<k<0$.

We wish to remove the weight $g(k)$.  We have the standard estimate 
$$g(k)\ll \min\left(N,\frac{1}{\|k/q\|}\right)=\min\left(N,\frac{q}{k}\right).$$
In addition 
$$g'(k)=-2\pi i\twosum{n<N}{n+M\in I}\frac{n}{q}e_q(-kn)\ll \frac{N}{q}\min\left(N,\frac{q}{k}\right).$$
We will split the sum over $k$ into intervals on which we may remove $g(k)$ by partial summation.  Specifically, let $K=[q/N]$ and  
$$S(r)=\max_{0\leq L\leq K}\left|\sum_{(r-1)K<k\leq
    (r-1)K+L}e_q(-Mk)S(a,k,q)\right| \;\;\;\;\;(r=1,2,3,\ldots).$$ 
Summing by parts we get, for any $K'\leq K$ that 
$$\sum_{(r-1)K\leq k\leq (r-1)K+K'}g(k)e_q(-Mk)S(a,k,q)\ll S(r)\min\left(N,\frac{q}{(r-1)K}\right)\ll \frac{N}{r}S(r).$$
It is therefore sufficient to estimate 
$$\frac{N}{q}\sum_{r\ll N}\frac{S(r)}{r}$$
which we accomplish by bounding each $S(r)$ individually.  We will prove the following, which easily implies Theorem \ref{shortkloost}.

\begin{lem}\label{lem:Sr}
Under the hypotheses of Theorem \ref{shortkloost} and with $K,S(r)$ as above we have 
$$S(r)\ll_{\epsilon, l} q^{1/2+\epsilon}\left(\sum_{j=1}^lK^{1-2^{-j}}q_{l-j+1}^{2^{-j}}+K^{1-2^{-l}}q_0^{1/2^{l+1}}+Kq_0^{-1/2^{l+1}}\right).$$
\end{lem}

\subsection{Differencing the Sum $S(r)$}

In the remainder of the paper we will frequently use without comment the fact that, since $q$ is squarefree, any  pair of integers $q',q''$ with $q'q''|q$ must be coprime.  We now apply a $q$-analogue of the van der Corput $A$-process.  Let $J$ be an interval whose length is bounded above by $K$.  Suppose $(a,q)=1$ and $s_1,\ldots, s_j$ are integers, for some $j\geq 1$.  We consider the more general sum 
$$T=\sum_{k\in J}e_q(-Mk)\prod_{i=1}^j S(a,k+s_i,q)$$
in which the value of $M$ may differ from that in $S(r)$.  The sums
$S(r)$ correspond to the case $j=1$ and $s_1=0$ of this.  The
following lemma describes a single van der Corput differencing step
applied to the sum $T$.  Note that the quantities $q,q_0,q_1$ occurring need not correspond to those in Theorem \ref{shortkloost}. 

\begin{lem}\label{onediff}
Suppose $q=q_0q_1$ with $q_1\leq K$, $(q_0,q_1)=1$ and $(a,q)=1$.  We have 
$$T^2\ll_{\epsilon,j} q^\epsilon q_1^{j+1}\left(Kq_0^j+\sum_{0<|h|\leq K/q_1}\left|\sum_{k\in J(h)}\prod_{i=1}^jS(a',k+s_i,q_0)S(a',k+q_1h+s_i,q_0)\right|\right)$$
where $J(h)$ is an interval of length at most $K$ which depends on $h$, and where $a'=a(\overline q_1)^2$.  
\end{lem}

\begin{proof}
We let $H=[K/q_1]\geq 1$ and 
$$a_k=\begin{cases}
e_q(-Mk)\prod_{i=1}^jS(a,k+s_i,q) & k\in J\\
0 & k\notin J.\\
\end{cases}$$
Since $H\geq 1$ we have 
\begin{eqnarray*}
T&=&\sum_k a_k\\
&=&\frac{1}{H}\sum_{h=1}^H\sum_k a_{k+q_1h}\\
&=&\frac{1}{H}\sum_k \sum_{h=1}^H a_{k+q_1h}.\\
\end{eqnarray*}
If $k+q_1h\in J$ then 
\begin{eqnarray*}
a_{k+q_1h}&=&e_q(-M(k+q_1h))\prod_{i=1}^jS(a,k+q_1h+s_i,q)\\
&=&e_q(-Mk)e_q(-Mq_1h)\prod_{i=1}^jS(a\overline q_1,(k+q_1h+s_i)\overline q_1,q_0)S(a\overline q_0,(k+s_i)\overline q_0,q_1).\\
\end{eqnarray*}
Since $q_1H\leq K$ the sum over $k$ is supported on an interval of length bounded by $O(K)$. By the Weil bound we get 
$$S(a\overline q_0,(k+s_i)\overline q_0,q_1)\ll_\epsilon q_1^{1/2+\epsilon}.$$
Therefore, applying Cauchy's inequality, we obtain 
$$H^2T^2\ll_{\epsilon,j} Kq_1^{j+\epsilon}\sum_k\left|\twosum{h=1}{k+q_1h\in J}^He_q(-Mq_1h)\prod_{i=1}^j S(a\overline q_1,(k+q_1h+s_i)\overline q_1,q_0)\right|^2.$$
Letting $a'=a(\overline q_1)^2$, as in the statement of the lemma,  we have $(a',q)=1$ and 
$$H^2T^2\ll_{\epsilon,j} Kq_1^{j+\epsilon}\sum_k\left|\twosum{h=1}{k+q_1h\in J}^He_q(-Mq_1h)\prod_{i=1}^j S(a',k+q_1h+s_i,q_0)\right|^2.$$ 
Expanding the square and reordering we deduce that 
\begin{eqnarray*}
H^2T^2&\ll_{\epsilon,j}&Kq_1^{j+\epsilon}\sum_{h_1,h_2=1}^H\left|\twosum{k}{k+q_1h_1,k+q_1h_2\in J}\prod_{i=1}^jS(a',k+q_1h_1+s_i,q_0)S(a',k+q_1h_2+s_i,q_0)\right|\\
&=&Kq_1^{j+\epsilon}\sum_{h_1,h_2=1}^H\left|\twosum{k\in J}{k+q_1(h_2-h_1)\in J}\prod_{i=1}^jS(a',k+s_i,q_0)S(a',k+q_1(h_2-h_1)+s_i,q_0)\right|\\
&\leq&KHq_1^{j+\epsilon}\sum_{|h|\leq H}\left|\twosum{k\in J}{k+q_1h\in J}\prod_{i=1}^jS(a',k+s_i,q_0)S(a',k+q_1h+s_i,q_0)\right|.\\
\end{eqnarray*}
We bound the $h=0$ term using the Weil bound on the individual Kloosterman sums to get 
$$HT^2\ll_{\epsilon,j} Kq_1^{j+\epsilon}\left(Kq_0^{j+\epsilon}+\sum_{0<|h|\leq H}\left|\twosum{k\in J}{k+q_1h\in J}\prod_{i=1}^jS(a',k+s_i,q_0)S(a',k+q_1h+s_i,q_0)\right|\right).$$
The result follows.  
\end{proof}

The previous lemma bounds $T$ in terms of sums with twice as many
Kloosterman factors. The new shifts are
$s_1,\ldots,s_j,s_1+q_1h,\ldots,s_j+q_1h$, and  
the exponential $e_q(-kM)$, if it exists, is removed.
We will apply it $l$ times, starting at the sum 
$$T=\sum_{k\in J}e_q(-Mk)S(a,k,q).$$
For the remainder of the paper $T$ will refer to this particular $j=1$ case of the above $T$ whereas $T(\ldots)$ will be one of the more general sums.

\begin{lem}\label{lem:ldiff}
Let $q$ be as in Theorem \ref{shortkloost} and $T$ as defined above.  We have
\begin{eqnarray*}
T^{2^l}&\ll_{\epsilon,l}&q^\epsilon\left(q^{2^{l-1}}\sum_{j=1}^lK^{2^l-2^{l-j}}q_{l-j+1}^{2^{l-j}}\right.\\ 
&&\hspace{2cm}+\left.K^{2^l-l-1}\left(q/q_0\right)^{2^{l-1}+1}\sum_{0<|h_1|\leq K/q_1}\ldots\sum_{0<|h_l|\leq K/q_l}|T(h_1,\ldots,h_l)|\right)\\ 
\end{eqnarray*}
where
$$T(h_1,\ldots,h_l)=\sum_{k\in J(h_1,\ldots,h_l)}\prod_{I\subseteq \{1,\ldots,l\}}S\left(a',k+\sum_{i\in I}q_ih_i,q_0\right),$$
with $J(h_1,\ldots,h_l)$ an interval of length at most $K$ and $(a',q)=1$.  
\end{lem}

\begin{proof}
Observe that by our assumption that $[q/N]\geq q_i$  we know that $K\geq q_i$ for $1\leq i\leq l$.  This means that the applications of Lemma \ref{onediff} in the following proof are all justified.

We use induction in $l$.  If $l=1$ then applying Lemma \ref{onediff} gives
$$T^2\ll_\epsilon q^\epsilon\left(qKq_1+q_1^2\sum_{0<|h_1|\leq K/q_1}\left|\sum_{k\in J(h_1)}S(a',k,q_0)S(a',k+q_1h_1,q_0)\right|\right),$$
as required.

Now suppose $l>1$ and that the result holds for $l-1$.  We assume that $q=q_0q_1\ldots q_l$ and apply the inductive hypothesis with the factorisation 
$$q=r_0r_1\ldots r_{l-1}$$
where $r_0=q_0q_1$, and $r_i=q_{i+1}$ for $1\leq i\leq l-1$.  This results in 
\begin{eqnarray*}
T^{2^{l-1}}&\ll_{\epsilon,l} &q^\epsilon\left(q^{2^{l-2}}
\sum_{j=1}^{l-1}K^{2^{l-1}-2^{l-1-j}}q_{l-j+1}^{2^{l-1-j}}\right.\\
&&\hspace{2cm}+\left.K^{2^{l-1}-l}\left(q/q_0q_1\right)^{2^{l-2}+1}\sum_{0<|h_2|\leq
  K/q_2}\ldots\sum_{0<|h_l|\leq K/q_l}|T(h_2,\ldots,h_l)|\right)\\ 
\end{eqnarray*}
with 
$$T(h_2,\ldots,h_l)=\sum_{k\in J(h_2,\ldots,h_l)}\prod_{I\subseteq \{2,\ldots,l\}}S\left(a',k+\sum_{i\in I}q_ih_i,q_0q_1\right).$$
Squaring our bound, using Cauchy's inequality on the final sum, we get 
\begin{eqnarray*}
T^{2^l}&\ll_{\epsilon,l}&q^\epsilon\left(q^{2^{l-1}}\sum_{j=1}^{l-1}K^{2^{l}-2^{l-j}}q_{l-j+1}^{2^{l-j}}\right.\\
&&\hspace{2cm}+\left.K^{2^{l}-l-1}\left(q/q_0q_1\right)^{2^{l-1}+1}\sum_{0<|h_2|\leq K/q_2}\ldots\sum_{0<|h_l|\leq K/q_l}|T(h_2,\ldots,h_l)|^2\right).\\ 
\end{eqnarray*}
We now use Lemma \ref{onediff} with $j=2^{l-1}$ to get the bound 
$$T(h_2,\ldots,h_l)^2\ll_{\epsilon,l} q^\epsilon q_1^{2^{l-1}+1}\left(Kq_0^{2^{l-1}}+\sum_{0<|h_1|\leq K/q_1}|T(h_1,\ldots,h_l)|\right)$$
where 
\begin{eqnarray*}
T(h_1,\ldots,h_l)&=&\sum_{k\in J(h_1,\ldots,h_l)}\prod_{I\subseteq \{2,\ldots,l\}}S\left(a'',k+\sum_{i\in I}q_ih_i,q_0\right)S\left(a'',k+\sum_{i\in I}q_ih_i+q_1h_1,q_0\right)\\
&=&\sum_{k\in J(h_1,\ldots,h_l)}\prod_{I\subseteq \{1,\ldots,l\}}S\left(a'',k+\sum_{i\in I}q_ih_i,q_0\right),\\
\end{eqnarray*}
for some $(a'',q)=1$.  Observe that this corresponds precisely to the $T(h_1,\ldots,h_l)$ given in the claim.

We conclude that 
\begin{eqnarray*}
T^{2^l}&\ll_{\epsilon,l}&q^\epsilon\left(q^{2^{l-1}}\sum_{j=1}^{l-1}K^{2^{l}-2^{l-j}}q_{l-j+1}^{2^{l-j}}+K^{2^{l}-1}\left(q/q_0q_1\right)^{2^{l-1}}q_1^{2^{l-1}+1}q_0^{2^{l-1}}\right.\\
&&\hspace{2cm}+\left.K^{2^{l}-l-1}\left(q/q_0\right)^{2^{l-1}+1}\sum_{0<|h_1|\leq
  K/q_1}\ldots\sum_{0<|h_l|\leq K/q_l}|T(h_1,\ldots,h_l)|\right)\\ 
&=&q^\epsilon\left(q^{2^{l-1}}\sum_{j=1}^lK^{2^{l}-2^{l-j}}q_{l-j+1}^{2^{l-j}}\right.\\ 
&&\hspace{2cm}+\left.K^{2^{l}-l-1}\left(q/q_0\right)^{2^{l-1}+1}\sum_{0<|h_1|\leq
  K/q_1}\ldots\sum_{0<|h_l|\leq K/q_l}|T(h_1,\ldots,h_l)|\right).\\ 
\end{eqnarray*}
\end{proof}

\subsection{Estimating $T(h_1,\ldots,h_l)$}

It remains to estimate 
$$T(h_1,\ldots,h_l)=\sum_{k\in J(h_1,\ldots,h_l)}\prod_{I\subseteq \{1,\ldots,l\}}S\left(a',k+\sum_{i\in I}q_ih_i,q_0\right),$$
where $(a',q_0)=1$.  

We begin with the following estimate for complete exponential sums to a prime modulus.

\begin{lem}\label{completeexp}
Let $p$ be a prime, $(a,p)=1$ and let $s_1,\ldots,s_j,b$ be integers.  We have 
$$\sum_{k\pmod p}e_p(-kb)\prod_{i=1}^jS(a,k+s_i,p)\ll_j \begin{cases}
p^{\frac{j+2}{2}} & b=0\text{ and }E(s_1,\ldots,s_j)\\
p^{\frac{j+1}{2}} & \text{otherwise},\\
\end{cases}$$ 
where $E(s_1,\ldots,s_j)$ denotes the property that all the $s_i$
occur with even multiplicity.
\end{lem}

\begin{proof}
The first part follows directly from the Weil bound 
$$|S(a,k+s_i,p)|\leq 2\sqrt{p}.$$
For the second part we use a result of Fouvry, Ganguly, Kowalski and Michel \cite[Proposition 3.2]{fgkm}.  Let 
$$\hypk(a;p)=\frac{S(a,1,p)}{p^{1/2}}.$$
If $k\not\equiv 0\pmod p$ then 
$$S(a,k,p)=p^{1/2}\hypk(ak;p).$$
We therefore have
$$\sum_{k\pmod p}e_p(-kb)\prod_{i=1}^jS(a,k+s_i,p)=p^{j/2}\twosum{k\pmod p}{k+s_i\not\equiv 0\pmod p}e_p(-kb)\prod_{i=1}^j\hypk(a(k+s_i));p)+O(jp^{j/2}),$$
where the error comes from the terms with $k+s_i\equiv0\pmod p$, for
which we can use the Weil bound.
The maps $k\mapsto a(k+s_i)$ are in $\mathrm{PGL}_2(\mathbb F_p)$.  If $s_i\ne s_j$ then $k\mapsto a(k+s_i),k\mapsto a(k+s_j)$ are different.

If $b=0$ then \cite[Proposition 3.2]{fgkm} states that if $\beta_1,\ldots,\beta_j\in \mathrm{PGL}_2(\mathbb F_p)$ then, provided the multiplicities of the $\beta_i$ are not all even, we have 
$$\twosum{k\pmod p}{\beta_i k\ne 0,\infty}e_p(-kb)\prod_{i=1}^j \hypk(\beta_i k;p)\ll_j p^{1/2}.$$
If $b\ne 0$ then the same bound can be shown to hold for all choices of $\beta_i$.  The proof involves some small modifications to the argument from \cite{fgkm}, detailed by Kowalski in a private communication.   

In the case that the shifts $s_i$ are all distinct this lemma may also be deduced from a result of Fouvry, Michel, Rivat and S{\'a}rk{\"o}zy \cite[Lemma 2.1]{fmrs}.
\end{proof}

We use this in conjunction with the following combinatorial result.

\begin{lem}
Let $p\geq 3$ be prime and let $h_1,\ldots,h_l\in \mathbb F_p$.  Suppose that the $2^l$ sums 
$$\sum_{i\in I}h_i\text{ for }I\subseteq \{1,\ldots,l\}$$
form a list of elements all of whose entries have even multiplicities.  At least one of the $h_i$ must then be $0$.  
\end{lem}

\begin{proof}
Let $\omega=e_p(1)$ and consider the algebraic integer $\alpha\in \Z[\omega]$ given by 
$$\alpha=\prod_{i=1}^l(1+\omega^{h_i})=\sum_{I\subseteq \{1,\ldots,l\}}\omega^{\sum_{i\in I}h_i}.$$
Our assumption that the $\sum_{i\in I}h_i$ all have even multiplicities therefore implies that $\alpha$ is a sum of even multiples of powers of $\omega$.  In particular $2|N_{\Q(\omega)/\Q}(\alpha)$.

If $h_i\not\equiv 0\pmod p$ then since $p\geq 3$ it is well known that 
$$N_{\Q(\omega)/\Q}(1+\omega^{h_i})=1.$$
We therefore have a contradiction unless at least one of the $h_i$ is $0$.
\end{proof}

Combining the last two lemmas we immediately deduce the following.  

\begin{lem}
Let $p$ be prime, $(a,p)=1$ and let $h_1,\ldots,h_l,b$ be integers.  We have 
$$\sum_{k\pmod p}e_p(-kb)\prod_{I\subseteq \{1,\ldots,l\}}S\left(a,k+\sum_{i\in I}h_i,p\right)\ll_l p^{\frac{2^l+1}{2}}(p,b,\prod h_i)^{1/2}.$$
\end{lem}

Next we generalise this to squarefree moduli.

\begin{lem}
Let $q$ be squarefree, $(a,q)=1$ and let $h_1,\ldots,h_l,b$ be integers.  For any $\epsilon>0$ we have 
$$\sum_{k\pmod q}e_q(-kb)\prod_{I\subseteq \{1,\ldots,l\}}S\left(a,k+\sum_{i\in I}h_i,q\right)\ll_{\epsilon,l} q^{\frac{2^l+1}{2}+\epsilon}(q,b,\prod h_i)^{1/2}.$$
\end{lem}

\begin{proof}
The sum has a multiplicative property.  Specifically, if $(q_0,q_1)=1$ then 
\begin{eqnarray*}
\lefteqn{\sum_{k\pmod {q_0q_1}}e_{q_0q_1}(-kb)\prod_{I\subseteq \{1,\ldots,l\}}S\left(a,k+\sum_{i\in I}h_i,q_0q_1\right)}\\
&=&\twosum{k_0\pmod{q_0}}{k_1\pmod{q_1}}e_{q_0}(-bk_0\overline{q_1})e_{q_1}(-bk_1\overline{q_0})\prod_{I\subseteq\{1,\ldots,l\}}S(a,k_0q_1\overline{q_1}+k_1q_0\overline{q_0}+\sum_{i\in I}h_i,q_0q_1)\\
&=&\sum_{k_0\pmod{q_0}}e_{q_0}(-bk_0\overline{q_1})\prod_{I\subseteq\{1,\ldots,l\}}S\left(a\overline{q_1},(k_0q_1\overline{q_1}+\sum_{i\in I}h_i)\overline{q_1},q_0\right)\\
&&\hspace{2cm}\times \sum_{k_1\pmod{q_1}}e_{q_1}(-bk_1\overline{q_0})\prod_{I\subseteq\{1,\ldots,l\}}S\left(a\overline{q_0},(k_1q_0\overline{q_0}+\sum_{i\in I}h_i)\overline{q_0},q_1\right).\\
\end{eqnarray*}
It follows that if $q$ is squarefree we may factorise the sum as a product over $p|q$  of sums to modulus $p$.  Each sum may then be estimated using the last lemma.  The integers $b,h_i$ occurring in the factors are different to those in our sum to modulus $q$, however the changes are simply by multiplicative factors coprime to $q$.  It follows that each factor may be bounded by 
$$C_lp^{\frac{2^l+1}{2}}(p,b,\prod h_i)^{1/2}$$ 
for some constant $C_l$, whence our sum is bounded by 
$$\prod_{p|q}\left(C_lp^{\frac{2^l+1}{2}}(p,b,\prod h_i)^{1/2}\right)\ll_{\epsilon,l} q^{\frac{2^l+1}{2}+\epsilon}(q,b,\prod h_i)^{1/2}.$$
\end{proof}

We now return to our sum 
\begin{eqnarray*}
T(h_1,\ldots,h_l)&=&\sum_{k\in J(h_1,\ldots,h_l)}\prod_{I\subseteq \{1,\ldots,l\}}S\left(a',k+\sum_{i\in I}q_ih_i,q_0\right)\\
&=&\sum_{j\pmod{q_0}}\twosum{k\in J(h_1,\ldots,h_l)}{k\equiv j\pmod{q_0}}\prod_{I\subseteq \{1,\ldots,l\}}S\left(a',j+\sum_{i\in I}q_ih_i,q_0\right)\\
&=&\frac{1}{q_0}\sum_{b\pmod{q_0}}h(b)\sum_{j\pmod{q_0}}e_{q_0}(bj)\prod_{I\subseteq \{1,\ldots,l\}}S\left(a',j+\sum_{i\in I}q_ih_i,q_0\right),\\
\end{eqnarray*}
with 
$$h(b)=\sum_{k\in J(h_1,\ldots,h_l)}e_{q_0}(-bk).$$
We have the standard estimate 
$$h(b)\ll \min\left(K,\frac{1}{\|b/q_0\|}\right).$$
Since $(q_i,q_0)=1$ for $i\ne 0$ we can write
$$(q_0,b,\prod q_ih_i)=(q_0,b,\prod h_i).$$
We may therefore use the last lemma to obtain 
$$T(h_1,\ldots,h_l)\ll_{\epsilon,l}q_0^{\frac{2^l+1}{2}+\epsilon}\cdot\frac{1}{q_0}\sum_{b\pmod{q_0}}\min\left(K,\frac{1}{\|b/q_0\|}\right)(q_0,b,\prod h_i)^{1/2}.$$
Finally we estimate 
\begin{eqnarray*}
\lefteqn{\frac{1}{q_0}\sum_{b\pmod{q_0}}\min\left(K,\frac{1}{\|b/q_0\|}\right)(q_0,b,\prod h_i)^{1/2}}\\
&\ll&\frac{K}{q_0}(q_0,\prod h_i)^{1/2}+\sum_{0<b\leq q_0/2}\frac{1}{b}(q_0,b,\prod h_i)^{1/2}\\
&=&\frac{K}{q_0}(q_0,\prod h_i)^{1/2}+\sum_{d|(q_0,\prod h_i)}d^{1/2}\twosum{0<b\leq q_0/2}{(q_0,b,\prod h_i)=d}\frac{1}{b}\\
&\ll_\epsilon&\frac{K}{q_0}(q_0,\prod h_i)^{1/2}+q^\epsilon\\
&\ll_\epsilon&q^\epsilon(q_0,\prod h_i)^{1/2}(\frac{K}{q_0}+1)\\
\end{eqnarray*}
so we conclude that 
$$T(h_1,\ldots,h_l)\ll_{\epsilon,l}q^\epsilon(\frac{K}{q_0}+1)q_0^{\frac{2^l+1}{2}}(q_0,\prod h_i)^{1/2}.$$

\subsection{Conclusion}

Inserting the above bound for $T(h_1,\ldots,h_l)$ into the result of Lemma \ref{lem:ldiff} we obtain 
\begin{eqnarray*}
T^{2^l}&\ll_{\epsilon,l}&q^\epsilon\left(q^{2^{l-1}}\sum_{j=1}^lK^{2^l-2^{l-j}}q_{l-j+1}^{2^{l-j}}\right.\\
&&\hspace{1cm}+\left.(\frac{K}{q_0}+1)K^{2^l-l-1}\left(q/q_0\right)^{2^{l-1}+1}q_0^{\frac{2^l+1}{2}}\sum_{0<|h_1|\leq
  K/q_1}\ldots\sum_{0<|h_l|\leq K/q_l}(q_0,\prod h_i)^{1/2}\right).\\  
\end{eqnarray*}
We have 
\begin{eqnarray*}
\sum_{0<|h_1|\leq K/q_1}\ldots\sum_{0<|h_l|\leq K/q_l}(q_0,\prod h_i)^{1/2}&\leq&\sum_{0<|h|\leq K^l/(q_1\ldots q_l)}\tau_l(h)(q_0,h)^{1/2}\\
&\ll_{\epsilon,l}&q^\epsilon\sum_{0<|h|\leq K^l/(q_1\ldots q_l)}(q_0,h)^{1/2}\\
&\ll_\epsilon&q^\epsilon K^l/(q_1\ldots q_l)\\
&=&q^{-1+\epsilon}K^lq_0.\\
\end{eqnarray*}
We conclude that 
$$T^{2^l}\ll_{\epsilon,l}q^{2^{l-1}+\epsilon}\left(\sum_{j=1}^lK^{2^l-2^{l-j}}q_{l-j+1}^{2^{l-j}}+(\frac{K}{q_0}+1)K^{2^l-1}q_0^{1/2}\right).$$
Lemma \ref{lem:Sr} now follows, on recalling that the sums $S(r)$ were a special case of the sum $T$, and therefore Theorem \ref{shortkloost} is proven.

\addcontentsline{toc}{section}{References} 
\bibliographystyle{plain}
\bibliography{../biblio}

\bigskip
\bigskip

Mathematical Institute,

University of Oxford,

Andrew Wiles Building, 

Radcliffe Observatory Quarter, 

Woodstock Road, 

Oxford 

OX2 6GG 

UK
\bigskip

{\tt irving@maths.ox.ac.uk}

\end{document}